\documentclass[12pt]{amsart}

\usepackage{amsrefs}
 \newtheorem{thm}{Theorem}[section]
 \newtheorem{cor}[thm]{Corollary}
 \newtheorem{lem}[thm]{Lemma}
 \newtheorem{prop}[thm]{Proposition}
 \theoremstyle{definition}
 
 \theoremstyle{remark}
 
 \newtheorem{ex}[thm]{Example}
 \newtheorem*{question}{Question}
 \numberwithin{equation}{section}

\begin{document}

\title[A characterization for cs Toeplitz operators]
 {A characterization for complex symmetric Toeplitz operators}

%----------Author 1
\author[Marcos S. Ferreira]{Marcos S. Ferreira}

\address{%
Departamento de Ciências Exatas e Tecnológicas\\
Universidade Estadual de Santa Cruz\\
Ilhéus, Bahia, Brasil}

\email{msferreira@uesc.br}

%\thanks{This work was completed with the support of our
%----------Author 2
%\author{A Second Author}
%\address{The address of\br
%the second author\br
%sitting somewhere\br
%in the world}
%email{dont@know.who.knows}


\date{July 13, 2022}

\begin{abstract}
In this paper we use orthonormal basis for the Hardy space $H^{2}(\mathbb{T})$, formed by rational functions, to characterize complex symmetric Toeplitz operators on $H^{2}(\mathbb{T})$. As a result, we get examples of these operators whose symbols are non-trigonometric functions.
\end{abstract}

%%% ----------------------------------------------------------------------
\maketitle
%%% ----------------------------------------------------------------------
%\tableofcontents

\section{Introduction and background}

A \emph{conjugation} $C$ on a separable complex Hilbert space $\mathcal H$ is an antilinear operator $C:\mathcal H\rightarrow\mathcal H$ such that:
\begin{enumerate}
  \item [(a)] $C$ is \emph{isometric}: $\left\langle Cf,Cg\right\rangle=\left\langle g,f\right\rangle$, $\forall f,g\in\mathcal H$.
  \item [(b)] $C$ is \emph{involutive}: $C^{2}=I$.
\end{enumerate}

A bounded linear operator $T$ on $\mathcal{H}$ is said to be \emph{complex symmetric} if there exists a conjugation $C$ on $\mathcal{H}$ such that $CT=T^{*}C$, where $T^{*}$ is the adjoint of $T$. We will often say that $T$ is $C$-symmetric. 

The concept of complex symmetric operators on separable Hilbert spaces is a natural generalization of complex symmetric matrices and their general study was initiated by Garcia, Putinar and Wogen \cite{Garcia,Garcia2,Garcia3,Garcia4}. The class of complex symmetric operators includes other basic classes of operators such as normal, Hankel, compressed Toeplitz and some Volterra operators.

An equivalent definition for the complex symmetric operator concept that will be useful in this work is:

\begin{prop}(\cite[Proposition 2]{Garcia})\label{prop1}
A bounded linear operator $T$ on $\mathcal{H}$ is $C$-symmetric if, and only if, there exists an orthonormal basis $\left\{e_{n}\right\}$ of $\mathcal{H}$ with respect to which $T$ has a symmetric matrix representation. 
\end{prop}

In Proposition \ref{prop1}, the conjugation $C$ that makes $T$ a complex symmetric operator is given by $Ce_{n}=e_{n}$ for all $n$ (see \cite[Lemma 1]{Garcia} for more details).

Throughout the paper, we assume that $\mathbb{T}$ is the boundary of the open unit disk $\mathbb{D}$ in the complex plane $\mathbb{C}$. Let $L^{2}(\mathbb{T})$ be the space of
square integrable functions on $\mathbb{T}$ with the inner product defined by
$$
\langle u,v \rangle=\int_{\mathbb{T}}u\overline{v}ds.
$$

Let $H^{2}(\mathbb{T})$ denote the classical \emph{Hardy space} associated to $\mathbb{D}$ which is the space of holomorphic functions on $\mathbb{D}$ with $L^{2}(\mathbb{T})$-boundary values in $\mathbb{T}$. Since the set of monomials $\left\{\frac{1}{\sqrt{2\pi}}z^{n}:n=0,1,2,\ldots\right\}$ is an orthonormal basis for $H^{2}(\mathbb{T})$, we have that $f\in H^{2}(\mathbb{T})$ if and only if
$$
f(z)=\sum_{n=0}^{\infty}a_{n}\frac{1}{\sqrt{2\pi}}z^{n} \ \text{where} \ \sum_{n=0}^{\infty}|a_{n}|^{2}<\infty.
$$

Let $L^{\infty}(\mathbb{T})$ be the space of essentially bounded measurable functions on $\mathbb{T}$. For each $\varphi$ be in $L^{\infty}(\mathbb{T})$, the \emph{Toeplitz operator} $T_{\varphi}:H^{2}(\mathbb{T})\rightarrow H^{2}(\mathbb{T})$, with symbol $\varphi$, is defined by
$$
T_{\varphi}f=P(\varphi f),
$$
for all $f\in H^{2}(\mathbb{T})$, where $P:L^{2}(\mathbb{T})\rightarrow H^{2}(\mathbb{T})$ is the orthogonal projection. It is clear that $T_{\varphi}$ is a bounded linear operator and that $T^{*}_{\varphi}=T_{\overline{\varphi}}$.

The study of complex symmetric Toeplitz operators is relatively recent and provides deep and important connections in several problems in quantum mechanics \cite{Bender,Garcia5}. One of the first examples of complex symmetric Toeplitz operator is due to Guo and Zhu \cite{Guo}. In this work, Guo and Zhu raised the question of characterizing complex symmetric Toeplitz operators on the Hardy space $H^{2}(\mathbb{T})$.

An interesting characterization of complex symmetric Toeplitz operators is due to Ko and Lee. In \cite{Ko} the authors considered the family of conjugations $C_{\lambda}$ on $H^{2}(\mathbb{T})$ given by
\begin{eqnarray}\label{eq2}
C_{\lambda}f(z)=\overline{f(\lambda\overline{z})}
\end{eqnarray}
with $\lambda\in\mathbb{T}$ and proved the following:

\begin{thm}
If $\varphi(z)=\sum_{n=-\infty}^{\infty}\widehat{\varphi}(n)z^{n}\in L^{\infty}(\mathbb{T})$, then $T_{\varphi}$ is $C_{\lambda}$-symmetric if, and only if, $\widehat{\varphi}(-n)=\lambda^{n}\widehat{\varphi}(n)$ for all $n\in\mathbb{Z}$, with $\lambda\in\mathbb{T}$.
\end{thm}

Recently, several authors have dedicated themselves to the study of complex symmetric Toeplitz operators on the Hardy space. Specifically regarding the characterization of these operators, we highlight \cite{Li, Bu, Arup, Chen}.

The plan of this paper is to obtain a characterization for complex symmetric Toeplitz operators (Theorem \ref{teo1}) using Proposition \ref{prop1}, that is, finding an orthonormal basis for $H^{2}(\mathbb{T})$ so that the Toeplitz operator $T_{\varphi}$ has a symmetric matrix representation. As a consequence we provide an example of complex symmetric Toeplitz operator $T_{\varphi}$ with non-trigonometric symbol.

\section{Complex symmetry of Toeplitz operators}

\subsection{$\mathfrak{J}$-symmetric Toeplitz operators}

Orthonormal basis play an important role in the study of complex symmetric operators (Proposition \ref{prop1}). In addition to this alternative definition, we use orthonormal basis to find conjugations on Hilbert spaces. In fact, if $\left\{f_{n}\right\}_{n\geq0}$ is an orthonormal basis for $H^{2}(\mathbb{T})$, then the antilinear operator
\begin{eqnarray}\label{eq5}
\sum_{n=0}^{\infty}a_{n}f_{n}\longmapsto\sum_{n=0}^{\infty}\overline{a_{n}}f_{n}
\end{eqnarray}
is naturally a conjugation on $H^{2}(\mathbb{T})$.

Let $p\in\mathbb{D}$. For each non-negative integer $n$, considering the rational function $R_{n}$ given by
$$
R_{n}=\sqrt{\frac{1-|p|^{2}}{2\pi}}\frac{(z-p)^{n}}{(1-\overline{p}z)^{n+1}},
$$
we have that $\mathfrak{B}=\left\{R_{n}:n=0,1,2,\ldots\right\}$ is an orthonormal basis for $H^{2}(\mathbb{T})$ (see \cite{Chung} for more details). Thus, we get from \eqref{eq5} the following:

\begin{lem}
Let $\mathfrak{J}:H^{2}(\mathbb{T})\rightarrow H^{2}(\mathbb{T})$ be defined by
\begin{eqnarray}\label{eq3}
\mathfrak{J}\left(\sum_{n=0}^{\infty}a_{n}\frac{(z-p)^{n}}{(1-\overline{p}z)^{n+1}}\right)=\sum_{n=0}^{\infty}\overline{a_{n}}\frac{(z-p)^{n}}{(1-\overline{p}z)^{n+1}}.
\end{eqnarray}
Then $\mathfrak{J}$ is a conjugation on $H^{2}(\mathbb{T})$.
\end{lem}

Below we present our main result:

\begin{thm}\label{teo1}
Let $\varphi(z)=\sum_{n=-\infty}^{\infty}\widehat{\varphi}(n)z^{n}\in L^{\infty}(\mathbb{T})$ and $p\in\mathbb{D}$. The following statements are equivalent:
\begin{enumerate}
    \item [(i)] $T_{\varphi}$ is $\mathfrak{J}$-symmetric.
    \item [(ii)] For all non-negative integer $k$, holds
    $$
\widehat{\varphi}(k)=\widehat{\varphi}(-k)+\overline{p}\left\{\widehat{\varphi}(-k-1)-\widehat{\varphi}(k-1)\right\}.
$$
    \item [(iii)] $\varphi(z)=\varphi(\overline{z})\dfrac{1+\overline{p}\overline{z}}{1+\overline{p}z}$.
\end{enumerate}
\end{thm}
\begin{proof}
$(i)\Leftrightarrow(ii)$ By \cite[Theorem 2.8]{Chung2}, for non-negative integers $m$ and $l$, the entry $t_{ml}$ of the matrix $[T_{\varphi}]_{\mathfrak{B}}$ is given by
$$
t_{ml}=\frac{1}{\sqrt{2\pi(1-|p|^{2})}}\widehat{\varphi}(m-l)+\frac{\overline{p}}{\sqrt{2\pi(1-|p|^{2})}}\widehat{\varphi}(m-l-1).
$$
Thus $[T_{\varphi}]_{\mathfrak{B}}$ is symmetric if, and only if
$$
\frac{1}{\sqrt{2\pi(1-|p|^{2})}}\left\{\widehat{\varphi}(m-l)+\overline{p}\widehat{\varphi}(m-l-1)\right\}
$$
is equal to
$$
\frac{1}{\sqrt{2\pi(1-|p|^{2})}}\left\{\widehat{\varphi}(l-m)+\overline{p}\widehat{\varphi}(l-m-1)\right\}
$$
which is equivalent to the desired equality just consider $k=m-l$. The equivalence follows from Proposition \ref{prop1}. 

$(ii)\Rightarrow(iii)$ Note that
\begin{equation}\label{eq7}
    \begin{split}
        \sqrt{2\pi}\left\{\varphi(z)-\varphi(\overline{z})\right\} 
        & = \sum_{k=-\infty}^{\infty}\widehat{\varphi}(k)z^{k}-\sum_{k=-\infty}^{\infty}\widehat{\varphi}(k)\overline{z}^{k} \\
        & = \sum_{k=-\infty}^{\infty}[\widehat{\varphi}(k)-\widehat{\varphi}(-k)]z^{k} \\
        & = \sum_{k=-\infty}^{\infty}\overline{p}\widehat{\varphi}(-k-1)z^{k}-\sum_{k=-\infty}^{\infty}\overline{p}\widehat{\varphi}(k-1)z^{k} \\
        & = \sum_{n=-\infty}^{\infty}\overline{p}\widehat{\varphi}(n)z^{-n-1}-\sum_{j=-\infty}^{\infty}\overline{p}\widehat{\varphi}(j)z^{j+1} \\
        & = \sqrt{2\pi}\left\{\overline{p}\varphi(\overline{z})\overline{z}-\overline{p}\varphi(z)z\right\}
    \end{split}
\end{equation}
and so
$$
\varphi(z)+\overline{p}z\varphi(z)=\varphi(\overline{z})+\overline{p}\overline{z}\varphi(\overline{z})
$$
as wished. 

$(iii)\Rightarrow(ii)$ Assuming (iii), we have
$$
\varphi(z)-\varphi(\overline{z})=\varphi(\overline{z})\overline{p}\overline{z}-\varphi(z)\overline{p}z
$$
and therefore, it follows from \eqref{eq7} that
\begin{equation*}
    \begin{split}
        \varphi(z)-\varphi(\overline{z})
        & = \sum_{k=-\infty}^{\infty}[\widehat{\varphi}(k)-\widehat{\varphi}(-k)]\frac{z^{k}}{\sqrt{2\pi}}
    \end{split}
\end{equation*}
and
\begin{equation*}
    \begin{split}
        \varphi(\overline{z})\overline{p}\overline{z}-\varphi(z)\overline{p}z 
        & = \sum_{l=-\infty}^{\infty}\overline{p}\widehat{\varphi}(l)\frac{z^{-l-1}}{\sqrt{2\pi}}-\sum_{j=-\infty}^{\infty}\overline{p}\widehat{\varphi}(j)\frac{z^{j+1}}{\sqrt{2\pi}} \\
        & = \sum_{k=-\infty}^{\infty}\overline{p}\left\{\widehat{\varphi}(-k-1)-\widehat{\varphi}(k-1)\right\}\frac{z^{k}}{\sqrt{2\pi}}.
    \end{split}
\end{equation*}
Thus since $\left\{\frac{1}{\sqrt{2\pi}}z^{n}:n\in\mathbb{Z}\right\}$ is an orthonormal basis for $L^{2}(\mathbb{T})$, we get (ii).
\end{proof}

Unlike the conjugations given in \eqref{eq2}, see \cite[Theorem 3.1]{Ko}, there are no complex symmetric Toeplitz operators $T_\varphi$ with conjugation $\mathfrak{J}$ with finite symbol $\varphi$.

\begin{cor}
Let $p\in\mathbb{D}$ nonzero. If $\varphi(z)=\sum_{n=-M}^{N}\widehat{\varphi}(n)z^{n}$ where $N\geq M>0$ with nonzero $\widehat{\varphi}(-M),\widehat{\varphi}(N)$, then $T_{\varphi}$ is not $\mathfrak{J}$-symmetric.
\end{cor}
\begin{proof}
If $T_{\varphi}$ is $\mathfrak{J}$-symmetric, then by Theorem \ref{teo1} (ii) we have
$$
\widehat{\varphi}(N+1)=\widehat{\varphi}(-N-1)+\overline{p}\left\{\widehat{\varphi}(-N-2)-\widehat{\varphi}(N)\right\}.
$$
But since $\widehat{\varphi}(N+1), \ \widehat{\varphi}(-N-1)$ and $\widehat{\varphi}(-N-2)$ are all zero, it follows that
$$
0=\widehat{\varphi}(N+1)=-\overline{p}\widehat{\varphi}(N)\neq0
$$
which is a contradiction.
\end{proof}

Next, we use Theorem \ref{teo1} to obtain examples of $\mathfrak{J}$-symmetric Toeplitz operators.

\begin{ex}\label{ex1}
Let $p\in(-1,1)$ a real number. If $\varphi(z)=\dfrac{1+p\overline{z}}{|1+p\overline{z}|}$, then $T_{\varphi}$ is $\mathfrak{J}$-symmetric. In fact, since
$$
\varphi(\overline{z})\cdot\dfrac{(1+p\overline{z})}{1+pz}=\frac{1+pz}{|1+pz|}\cdot\frac{1+p\overline{z}}{1+pz}=\frac{1+p\overline{z}}{|1+p\overline{z}|}=\varphi(z)
$$
we have by Theorem \ref{teo1} that $T_{\varphi}$ is $\mathfrak{J}$-symmetric.
\end{ex}

\subsection{An example of a $\mathfrak{C}_{p}$-symmetric Toeplitz operator}

Let's consider another family of conjugation not as expected as the conjugation $\mathfrak{J}$. Let $p\in\mathbb{D}$. In \cite{Ferreira}, Ferreira and De Assis Júnior showed that the antilinear operator
\begin{eqnarray}\label{eq4}
\sum_{n=0}^{\infty}a_{n}z^{n}\longmapsto\sqrt{1-|p|^{2}}\sum_{n=0}^{\infty}\overline{a_{n}}\frac{(z-p)^{n}}{(1-\overline{p}z)^{n+1}}
\end{eqnarray}
is a conjugation on $H^{2}(\mathbb{T})$. Of course, since
\begin{eqnarray*}
\left\{\sqrt{\frac{1-|p|^{2}}{2\pi}}\frac{(p-z)^{n}}{(1-\overline{p}z)^{n+1}}:n=0,1,2,\ldots\right\}
\end{eqnarray*}
is also an orthonormal basis on $H^{2}(\mathbb{T})$, we have for the same reason as \eqref{eq4} that
\begin{eqnarray}\label{eq6}
\mathfrak{C}_{p}\left(\sum_{n=0}^{\infty}a_{n}z^{n}\right)=\sqrt{1-|p|^{2}}\sum_{n=0}^{\infty}\overline{a_{n}}\frac{(p-z)^{n}}{(1-\overline{p}z)^{n+1}}
\end{eqnarray}
is a conjugation on $H^{2}(\mathbb{T})$.

Now, let $\phi$ be an analytic map from $\mathbb{D}$ into itself and $\psi$ be an analytic map on $\mathbb{D}$. Remember that the \textit{weighted composition operator} $C_{\psi,\phi}$ is defined by
$$
C_{\psi,\phi}h(z)=\psi(z)h(\phi(z))
$$
where $h$ is an analytic map on $\mathbb{D}$. Considering $\lambda=1$ in \eqref{eq2}, $p\in(-1,1)$ a real number and 
$$
\psi(z)=\frac{\sqrt{1-|p|^{2}}}{1-pz}
$$
and
$$
\phi(z)=\frac{p-z}{1-pz}
$$
Fatehi \cite{Fatehi} showed that $C_{\psi,\phi}C_{1}$ is a conjugation over $H^{2}(\mathbb{T})$. Since, for each $f\in H^{2}(\mathbb{T})$, we have
\begin{equation*}
    \begin{split}
        C_{\psi,\phi}C_{1}f(z) & = C_{\psi,\phi}C_{1}\left(\sum_{n=0}^{\infty}a_{n}z^{n}\right) \\
        & = C_{\psi,\phi}\left(\sum_{n=0}^{\infty}\overline{a_{n}}z^{n}\right) \\
        & = \frac{\sqrt{1-|p|^{2}}}{1-pz}\left(\sum_{n=0}^{\infty}\overline{a_{n}}\left[\frac{p-z}{1-pz}\right]^{n}\right) \\
        & = \sqrt{1-|p|^{2}}\sum_{n=0}^{\infty}\overline{a_{n}}\frac{(p-z)^{n}}{(1-pz)^{n+1}} \\
        & = \mathfrak{C}_{p}f(z),
    \end{split}
\end{equation*}
and so we have the following:
\begin{ex}(\cite[Example 2.4]{Fatehi})\label{ex2}
Let $p\in(-1,1)$ a real number. If $\varphi(z)=\frac{1-pz}{|1-pz|}$, then $T_{\varphi}$ is $\mathfrak{C}_{p}$-symmetric by  
\end{ex}

Note that Examples \ref{ex1} and \ref{ex2} are very similar, although $T_{\frac{1-pz}{|1-pz|}}$ is not $\mathfrak{J}$-symmetric by Theorem \ref{teo1}. In this sense, the question is natural:

\begin{question}
Let $p\in\mathbb{D}$ and $\varphi\in L^{\infty}(\mathbb{T})$. Under what conditions is $T_{\varphi}$ complex symmetric with the conjugation $\mathfrak{C}_{p}$?
\end{question}

%\section*{Conflicts of interests/Competing interests}

\end{document}